\documentclass[11pt, reqno]{amsart}
\pdfoutput=1

\usepackage{amssymb}
\usepackage{mathtools}
\usepackage{mathrsfs}
\usepackage{stmaryrd}
\usepackage{upgreek}
\usepackage{centernot}
\usepackage{cancel}

\usepackage{bm}

\usepackage[T2A,T1]{fontenc}
\usepackage[utf8]{inputenc}
\usepackage[russian, french, english]{babel}
\usepackage{lmodern}
\usepackage[spacing=true,kerning=true,babel=true,nopatch=footnote]{microtype}
\usepackage[shortcuts]{extdash}
\usepackage[inline]{enumitem}
\usepackage{indentfirst}
\usepackage{xspace}
\usepackage{xcolor}
\usepackage{csquotes}

\usepackage[a4paper, left=2cm, right=2cm, top=1in, bottom=1in]{geometry}
\usepackage{graphicx}
\usepackage{wrapfig}
\usepackage{float}
\usepackage{array}
\usepackage{multicol}
\usepackage{mdframed}
\usepackage[skins]{tcolorbox}
\usepackage{framed}
\usepackage[justification=centering, labelfont=bf, font=small]{caption}
\usepackage{algorithm}

\usepackage{tikz}
\usetikzlibrary{shapes}
\usetikzlibrary{arrows.meta}
\usetikzlibrary{decorations.pathmorphing}
\usetikzlibrary{patterns}

\usepackage{titlesec}
\usepackage{titletoc}

\usepackage[
    sortcites,
    backend=biber,
    style=alphabetic,
    sorting=nyt,
    maxnames=100,
    backref=false
]{biblatex}
\usepackage[hidelinks]{hyperref}
\hypersetup{
    pdftitle={The List Chromatic Number of Uncrowded Hypergraphs},
    pdfauthor={Jing Yu and Junchi Zhang},
    pdfsubject={List coloring of uncrowded uniform hypergraphs},
    pdfkeywords={list coloring, list chromatic number, uniform hypergraphs, girth, semi-random method}
}
\usepackage{cleveref}

\newcommand{\bemph}[1]{{\normalfont#1}}
\newcommand{\ep}[1]{\bemph{(}#1\bemph{)}}

\newtheoremstyle{bfnote}{}{}{\slshape}{}{\bfseries}{\bfseries.}{ }{\thmname{#1}\thmnumber{ #2}\thmnote{ \ep{\normalfont{}#3}}}
\theoremstyle{bfnote}
\newtheorem{theo}{Theorem}[section]
\newtheorem*{theo*}{Theorem}

\newtheorem{lemma}[theo]{Lemma}
\newtheorem{claim}[theo]{Claim}

\newtheorem{corl}[theo]{Corollary}
\newtheorem{conj}[theo]{Conjecture}

\newtheorem*{corl*}{Corollary}

\theoremstyle{definition}

\newtheorem*{defn*}{Definition}

\newtheorem*{exmp*}{Example}

\theoremstyle{remark}
\newtheorem*{ques*}{Question}
\newtheorem*{remk*}{Remark}

\crefname{theo}{theorem}{theorems}
\Crefname{theo}{Theorem}{Theorems}
\crefname{lemma}{lemma}{lemmas}
\Crefname{lemma}{Lemma}{Lemmas}
\crefname{claim}{claim}{claims}
\Crefname{claim}{Claim}{Claims}
\crefname{prop}{proposition}{propositions}
\Crefname{prop}{Proposition}{Propositions}
\crefname{corl}{corollary}{corollaries}
\Crefname{corl}{Corollary}{Corollaries}
\crefname{defn}{definition}{definitions}
\Crefname{defn}{Definition}{Definitions}
\crefname{remk}{remark}{remarks}
\Crefname{remk}{Remark}{Remarks}
\crefname{alg}{algorithm}{algorithms}
\Crefname{alg}{Algorithm}{Algorithms}

\newcommand*{\myproofname}{Proof}

\makeatletter
\newcommand{\neutralize}[1]{\expandafter\let\csname c@#1\endcsname\count@}
\makeatother

\newcommand{\set}[1]{\{#1\}}

\newcommand{\F}{\mathcal{F}}

\renewcommand{\P}{\mathbb{P}}
\newcommand{\E}{\mathbb{E}}
\renewcommand{\epsilon}{\varepsilon}
\newcommand{\eps}{\epsilon}
\renewcommand{\phi}{\varphi}
\renewcommand{\leq}{\leqslant}
\renewcommand{\geq}{\geqslant}
\newcommand{\defeq}{\coloneqq}

\newcommand{\emphd}[1]{{\fontseries{b}\selectfont\textsf{#1}}}

\newcommand{\dom}{\mathsf{dom}}

\newcommand{\DP}{\mathrm{DP}}

\DeclareMathOperator{\Safe}{Safe}

\numberwithin{equation}{section}

\titleformat{\section}[block]{\large\bfseries\sffamily}{\thesection.}{1ex}{}
\titleformat{\subsection}[block]{\bfseries\sffamily}{\thesubsection.}{1ex}{}
\titleformat{\subsubsection}[block]{\itshape}{\bfseries\upshape\sffamily\thesubsubsection.}{1ex}{}

\titlespacing*{\section}{0pt}{*3}{*1}
\titlespacing*{\subsection}{0pt}{*3}{*1}
\titlespacing*{\subsubsection}{0pt}{*2}{*1}

\titlecontents{section}[1.5em]{\smallskip}{\bfseries\thecontentslabel\hspace{1.02em}}{\bfseries}{\,\,\titlerule*[0.77pc]{}\bfseries\contentspage}
\titlecontents{subsection}[4em]{\smallskip}{\thecontentslabel\hspace{1.02em}}{\hspace*{2.32em}}{\,\,\titlerule*[0.77pc]{.}\contentspage}

\renewbibmacro{in:}{}
\renewbibmacro*{volume+number+eid}{\printfield{volume}\setunit*{\addnbspace}\printfield{number}\setunit{\addcomma\space}\printfield{eid}}
\DeclareFieldFormat[article]{volume}{\textbf{#1}\space}
\DeclareFieldFormat[article]{number}{\mkbibparens{#1}}
\DeclareFieldFormat{journaltitle}{#1,}
\DeclareFieldFormat[thesis]{title}{\mkbibemph{#1}\addperiod}
\DeclareFieldFormat[article, unpublished, thesis, incollection, inproceedings]{title}{\mkbibemph{#1},}
\DeclareFieldFormat[book]{title}{\mkbibemph{#1}\addperiod}
\DeclareFieldFormat[unpublished]{howpublished}{#1, }
\DeclareFieldFormat{pages}{#1}
\DeclareFieldFormat[article]{series}{Ser.~#1\addcomma}

\setlist{topsep=3pt,itemsep=3pt}



\title{\sffamily A polynomial criterion for DP-coloring of Uncrowded Hypergraphs}
\date{}

\author[J. Yu]{Jing~Yu}
\address{Shanghai Center for Mathematical Sciences, Fudan University, Shanghai 200438, China}
\email{jyu@fudan.edu.cn}
\thanks{The first author was partially supported by the National Natural Science Foundation of China under grant nos. 12371343 and 12525110 (PI: Hehui Wu).}
\author[J. Zhang]{Junchi~Zhang}
\address{Shanghai Center for Mathematical Sciences, Fudan University, Shanghai 200438, China}
\email{jczhang24@m.fudan.edu.cn}

\begin{document}

\maketitle

\begin{abstract}
We prove a color-degree criterion for DP-coloring of finite uncrowded hypergraphs of bounded rank. For a $q$-fold cover with edge sizes between $2$ and $r+1$, let $D_\ell$ be the maximum number of forbidden colorings on $(\ell+1)$-edges prescribing a fixed color at a fixed vertex. If $q$ is sufficiently large in terms of $r$, then the cover admits a coloring whenever
\[
 \sum_{\ell=1}^r\frac{D_\ell}{q^\ell}\leq\log q-20\log\log q.
\]
For uncrowded $(r+1)$-uniform hypergraphs of maximum degree $\Delta$, this gives a DP-chromatic number at most
$(1+O_r(\log\log\Delta/\log\Delta))(r\Delta/\log\Delta)^{1/r}$.
\end{abstract}

\section{Introduction}

Coloring graphs and hypergraphs under local sparsity assumptions is a central topic in probabilistic combinatorics. DP-coloring, also called correspondence coloring, was introduced for graphs by Dvo\v r\'ak and Postle~\cite{Postle2018corres} as a generalization of list coloring, and was extended to hypergraphs by Bernshteyn and Kostochka~\cite{Anton2019DP}. In this paper we prove a color-degree criterion for DP-coloring uncrowded hypergraphs, allowing edges of different sizes.

All hypergraphs are finite, and all logarithms are natural. A \emphd{hypergraph} is a pair $G=(V,E)$, where every edge is a subset of $V$ of size at least two. Its \emphd{chromatic number} $\chi(G)$ is the least number of colors needed to color its vertices without a monochromatic edge. Its \emphd{list chromatic number} $\chi_\ell(G)$ is the least integer $k$ such that every assignment of $k$-element lists to the vertices admits such a coloring from the assigned lists.

We use the following formulation of hypergraph DP-coloring. A \emphd{cover} of $G$ is a pair $(L_0,\F)$, where $L_0$ assigns a finite list of colors to each vertex and $\F$ is a family of maps satisfying:
\begin{itemize}
\item each $\psi\in\F$ has domain $e\in E$ and satisfies $\psi(v)\in L_0(v)$ for every $v\in e$;
\item if distinct $\psi,\psi'\in\F$ have the same domain $e$, then $\psi(v)\neq\psi'(v)$ for every $v\in e$.
\end{itemize}
The maps in $\F$ are the \emphd{forbidden colorings}, and the second condition is the \emphd{matching condition}. The cover is \emphd{$k$-fold} if $|L_0(v)|=k$ for every $v$. An \emphd{$(L_0,\F)$-coloring} is a function $f$ on $V$ such that
\[
 f(v)\in L_0(v)\quad(v\in V),\qquad
 f|_{\dom\psi}\neq\psi\quad(\psi\in\F).
\]
The hypergraph $G$ is \emphd{$k$-DP-colorable} if every $k$-fold cover of $G$ admits a coloring. Its \emphd{DP-chromatic number} $\chi_{\mathrm{DP}}(G)$ is the least such $k$. Ordinary list coloring is the special case in which the forbidden colorings of each edge are its available monochromatic colorings. Thus
\[
 \chi(G)\leq\chi_\ell(G)\leq\chi_{\mathrm{DP}}(G).
\]

Obtaining upper bounds for chromatic number of graphs or hypergraphs with local constraints has a long history.
For graphs, Johansson \cite{Johansson} proved that triangle-free graphs of maximum degree $\Delta$ have list chromatic number $O(\Delta/\log\Delta)$, extending the logarithmic improvement of Ajtai, Koml\'os and Szemer\'edi \cite{Ajtai1980ramsey} for triangle-free graphs to list colorings.
The best known leading constant for the list chromatic number is due to Molloy \cite{Molloy19list}, who improved the upper bound to $(1+o(1))\Delta/\log\Delta$. Later, Bernshteyn \cite{Anton2019Molloy} gave a short proof of the Johansson--Molloy theorem in the more general setting of DP-coloring. Related list chromatic number bounds for locally sparse graphs were obtained by Vu~\cite{Vu2002localsparse}.

When it comes to hypergraphs, we need to introduce some notions.
For $r\geq 1$, $G$ has \emphd{rank} $r+1$ if the largest size of edges in $E$ is $r+1$, and is called \emphd{$(r+1)$-uniform} if every edge has size $r+1$.
For $k\geq2$, a \emphd{Berge cycle of length $k$} consists of $k$ distinct vertices $v_1,\ldots,v_k$ and $k$ distinct edges $e_1,\ldots,e_k$ such that $\{v_i,v_{i+1}\}\subseteq e_i$, with indices taken modulo $k$.
A hypergraph is called \emphd{uncrowded} if it contains no Berge cycles of lengths $2$, $3$, or $4$.
In particular, it is \emphd{linear}: two distinct edges meet in at most one vertex.

Ajtai, Koml\'os, Pintz, Spencer, and Szemer\'edi~\cite{Ajtai1982uncrowded} proved that every $n$-vertex uncrowded $(r+1)$-uniform hypergraph of sufficiently large maximum degree $\Delta$ has an independent set of size at least $c_r n(\log\Delta/\Delta)^{1/r}$ for some $c_r>0$.
On the coloring side, Frieze and Mubayi~\cite{Mubayi13color} established the bound $O_r((\Delta/\log\Delta)^{1/r})$ for linear $(r+1)$-uniform hypergraphs for fixed $r\geq2$. For uncrowded hypergraphs, Iliopoulos~\cite{Ili21} proved the list-coloring bound $(1+o(1))r\left(\frac{\Delta}{\log\Delta}\right)^{1/r}$.
Li and Postle~\cite{Postle25chrom} proved that every triangle-free hypergraph $G$ of rank at most $r+1$ satisfies
\[
\chi_\ell(G)
=O_r\left(
\max_{2\leq j\leq r+1}
\left(
\frac{\Delta_j(G)}{\log\Delta_j(G)}
\right)^{1/(j-1)}
\right),
\]
where $\Delta_j(G)$ denotes the maximum number of $j$-edges containing a vertex, and $C_r$ depends only on $r$.
They also asked for analogous DP-coloring results and explained why their method does not readily extend to that setting. 
We refer to Kang, Kelly, K\"uhn, Methuku, and Osthus~\cite{Kang2023survey} for a survey of coloring results and nibble methods.

Now we begin to state our main result. 
Throughout, a pair $(e,\psi)$ is understood to satisfy $\psi\in\F$ and $\dom\,\psi=e$. 
Let $(L_0,\F)$ be a cover of some hypergraph $G = (V, E)$. For $v\in V$, $c\in L_0(v)$, and $1\leq\ell\leq r$, define 
\begin{align*}
    d_\ell(v,c) &=\bigl|\{(e,\psi) \,:\,  v\in e,\; |e|=\ell+1,\; \psi(v)=c\}\bigr|,\\
    D_\ell&=\max_{v\in V,\ c\in L_0(v)}d_\ell(v,c).
\end{align*}
An empty maximum is zero.
Using this notation, we can obtain the following condition for the existence of an $(L_0,\F)$-coloring.

\begin{theo}\label{thm:main}
    For every integer $r\geq 1$, there exists $q_0=q_0(r)>0$ such that for any $q\geq q_0$, an uncrowded hypergraph of rank at most $r+1$ with a $q$-fold cover $(L_0,\F)$ has an $(L_0,\F)$-coloring if 
    \begin{equation}\label{eq:colordegree}
    \sum_{\ell=1}^r\frac{D_\ell}{q^\ell} \leq\log q-20\log\log q.
    \end{equation}
\end{theo}
The constant $20$ is not optimized; for every fixed $\varepsilon>0$, the stronger assumption
\[
 \sum_{\ell=1}^r D_\ell/q^\ell\leq(1-\varepsilon)\log q
\]
is sufficient when $q$ is large enough in terms of $r$ and $\varepsilon$.
\Cref{thm:main} provides a criterion of DP-coloring in terms of degrees of all sizes for general non-uniform hypergraphs.
Under the stronger assumption of uncrowdedness, \Cref{thm:main} yields the corresponding bound for DP-coloring:
\[
\chi_{\DP}(G)
\leq C_r\left[
1+\max_{2\leq j\leq r+1}
\left(
\frac{\Delta_j(G)}{\log(2+\Delta_j(G))}
\right)^{1/(j-1)}
\right],
\]
which answers the question of Li and Postle in the uncrowded setting. In particular, for uniform hypergraphs, we have the following upper bound of DP-chromatic number.
For $r\geq 2$, the corresponding ordinary list-coloring bound was proved in the first version of the present preprint \cite{yu2026list} and in \cite{DMV26v3}. \Cref{cor:main} extends this bound to arbitrary DP-covers.

\begin{corl}\label{cor:main}
    For uncrowded $(r+1)$-uniform hypergraphs $H$ of maximum degree at most $\Delta$,
    \begin{equation*}
    \chi_{\DP}(H)\leq    \left(1+O_r\!\left(\frac{\log\log\Delta}{\log\Delta}\right)\right)
     \left(\frac{r\Delta}{\log (\Delta+1)}\right)^{1/r}.
    \end{equation*}
\end{corl}

For non-uniform hypergraphs, a weaker bound of the same scale can be obtained.

\begin{corl}\label{cor:color-degree}
    Fix an integer $r\geq1$ and $\varepsilon>0$. For all sufficiently large $d$ in terms of $r,\varepsilon$, the following holds.
    
Let $G$ be an uncrowded $(r+1)$-uniform hypergraph, and let $(L_0,\F)$ be a cover of $G$ with $D_r\leq d$.
If every list has size at least
\[
 \left\lceil(1+\varepsilon)\left(\frac{rd}{\log d}\right)^{1/r}\right\rceil,
\]
then $(L_0,\F)$ admits a coloring.
\end{corl}

\section*{Proof overview}
We prove the main theorem using an iterative semi-random coloring procedure.
At stage $i$, let $Q_i=(U_i,L_i,\varphi_i)$ denote the current partial coloring, where $U_i$ is the set of uncolored vertices, $L_i$ assigns their available color lists, and $\varphi_i$ colors the remaining vertices. We remove any color that would complete a forbidden assignment with only one uncolored vertex. A forbidden assignment is called \emphd{active} if it remains compatible with the partial coloring and the current lists, and has at least two uncolored vertices.
We maintain a common list size $s_i$ and write $D_{\ell,Q_i}$ for the maximum number of active assignments with $\ell+1$ uncolored vertices prescribing a fixed color at a fixed vertex.

Let $D_\ell$ denote the corresponding initial color-degrees. Set
\[
P(z)=\sum_{\ell=1}^r
\left(\frac{D_\ell}{q^\ell}+1\right)z^\ell,
\qquad
P_\ell(z)=\frac{P^{(\ell)}(z)}{\ell!}.
\]
Starting with $s_0=q$, $t_0=0$, and $x_0=1$, we maintain
\[
D_{\ell,Q_i}\leq (x_i s_i)^\ell P_\ell(t_i),
\qquad 1\leq\ell\leq r.
\]
Here $x_i s_i$ controls the scale of the residual color-degrees, while $t_i$ records the accumulation of smaller constraints.

Take $h=(\log q)^{-2}$. At stage $i$, each uncolored vertex independently makes a tentative color choice with probability
\[
\theta_i=\frac{h}{x_i}.
\]
Conditional on making a choice, the vertex selects a color uniformly from its current list of size $s_i$. Thus each individual color is chosen with probability
\[
p_i=\frac{\theta_i}{s_i}=\frac{h}{x_i s_i}.
\]
The choice of $p_i$ ensures that $x_i s_i p_i=h$, which makes the evolution of residual constraints correspond to a Taylor translation of $P$ with step size $h$.

A tentative color is accepted only if no active forbidden coloring prescribing that color has all its other uncolored vertices tentatively choosing their prescribed colors. The probability that a fixed color passes this safety test is at least
\[
\gamma_i=\prod_{\ell=1}^r
(1-p_i^\ell)^{(x_i s_i)^\ell P_\ell(t_i)}.
\]
Independent equalizing coins then ensure that each color belongs to its temporary list with probability exactly $\gamma_i$. After trimming, the remaining lists have common size
\[
s_{i+1}=\lfloor(1-h^2)\gamma_i s_i\rfloor.
\]

More precisely, the contributions of larger residual constraints to smaller ones are governed by the identity
\[
\sum_{m=\ell}^r
\binom{m}{\ell}P_m(t_i)h^{m-\ell}
=P_\ell(t_i+h).
\]
Using the independence properties of uncrowded hypergraphs, concentration inequalities, and the Lov\'asz Local Lemma, we preserve the degree bounds with
\[
t_{i+1}=t_i+h,\qquad
x_{i+1}=x_i-h+O_r(h/\log q).
\]
Moreover,
\[
-\log\gamma_i
=(1+O(p_i))\bigl(P(t_i+h)-P(t_i)\bigr),
\]
so the accumulated list losses are controlled by a telescoping sum.
The hypothesis of the theorem ensures that $s_i\geq(\log q)^{12}$ throughout the process. After $O((\log q)^2)$ rounds, $x_i$ is sufficiently small for a Rosenfeld-type counting argument to complete the coloring.

\subsection*{Note added} 
Throughout this note, references to our earlier proof are to the first version of the present preprint \cite{yu2026list}.

We clarify the public chronology and several technical comparisons made in the Note added to~\cite{dhawan2026list}, which is retained in the October~24, 2026 revision~\cite{DMV26v3}. The uniform list-coloring bound was stated in~\cite[Theorem~1.11]{dhawan2026ind}, posted on June~16, 2026, explicitly as a result to be proved in forthcoming work; that version did not contain its coloring proof. Our July~6, 2026 preprint~\cite{yu2026list} provided a complete proof, including the random nibble, the control of all residual edge sizes, the iteration, and the terminal counting argument. The coloring proof of~\cite{dhawan2026list} became publicly available in its September~7, 2026 revision. Our July version acknowledged their earlier announcement without comparing methods, since their proof was not yet public. We appreciate the explicit acknowledgment of our concurrent independent-set result restored in~\cite{DMV26v3}. However, the Note added concerning the list-coloring proofs remains unchanged, including its account of the public chronology and methodological differences. The earlier announcement and the order in which complete proofs became publicly available should not be conflated.
We do not seek to determine which unpublished proof was completed first, and simply clarify the public record.

The technical comparison in~\cite{DMV26v3} also understates several concrete correspondences.
They contrast their  color-degree tracking via four recursively defined quantities with our two-parameter approach. 
Indeed, these are different presentations of the residual coloring problem and \emph{have exactly the same binomial form}.
By putting $r=k-1$ and $j=\ell-1$ in~\cite[Definition~4.5, (Q3)]{dhawan2026list},
we obtain
\[
t_i\binom rj \left(\frac{\alpha\gamma_i}{q_i}\right)^{r-j}\beta_i^j=\binom rj\widetilde B_i^{\,r-j}
\widetilde d_i^{\,j/r},
\]
where
\[
\widetilde B_i=t_i^{1/r}\frac{\alpha\gamma_i}{q_i},
\qquad
\widetilde d_i=t_i\beta_i^r.
\]
This is exactly the binomial form of the degree bound in~\cite[(3.1)]{yu2026list}.

The authors of \cite{DMV26v3} also contrast their polynomially small concentration errors with the factor $(\log\Delta)^{-4}$ in our proof. 
But this factor is simply our choice for easier computation, not a limitation of Lemma~3.1.
For fixed $r,\eps$, we may instead take $\zeta=\Delta^{-c}$, where $0<c<\alpha/2$ and $\alpha=\beta/(4r)$, and set $S_\Delta=2K^2\Delta^{2c}\log\Delta$.
The bounds in Section~4 still satisfy the concentration conditions and control the accumulated errors.

In \cite[Section~2.1]{DMV26v3}, the authors emphasize summing over residual edge sizes rather than taking their maximum. This cross-layer control also appears in our July proof.
Already in~\cite[(3.3), (3.16)]{yu2026list}, we introduced and bounded the weighted cross-layer quantity
\[
\tau_i=\sum_{j=1}^r D_{j,Q_i}\,p_i^j
\le (B_i+h)^r-B_i^r.
\]
These increments were explicitly summed by telescoping in~\cite[(4.21)]{yu2026list}; the concluding Rosenfeld-type count in~\cite[Section~5]{yu2026list} likewise summed the
contributions from all residual layers.

The final qualification concerns the claimed differences in scope.
The statement in~\cite{DMV26v3} that our paper addresses only the maximum-degree list-coloring setting describes our stated theorem, but not the actual scope of our proof.
Our Lemma~3.1 and its dependency bounds were formulated in terms of the \emph{active color-degrees} $D_{j,Q}$.
For an initial $(r+1)$-uniform instance with local color-degree at most $d$, one has $D_{j,Q_0}=0$ for $j<r$, and the same proof applies with $D_{r,Q_0}\le d$ by initializing $B_0=0$ and $d_0=d$ and choosing parameters in terms of $d$.
Hence the uniform color-degree result is implicit in the July argument, even though it was not separately stated there.


\section{Preliminaries}\label{sec:pre}

Fix an uncrowded hypergraph $G=(V,E)$ and a $q$-fold cover $(L_0,\F)$. 
A \emphd{partial coloring} $Q=(U,L,\phi)$ of $G$ consists of an uncolored set $U$, color lists $L(v)\subseteq L_0(v)$ for $v\in U$, and a coloring $\phi$ on $V\setminus U$ such that for any $e\in E$ and $\psi\in\F$ with $e\subseteq V\setminus U$ and $\dom\; \psi=e$, we have $\varphi|_{e}\neq \psi$.
The partial coloring is \emphd{compatible} if, whenever $e\cap U=\{v\}$ and $\psi\in\F$ has domain $e$ and agrees with $\varphi$ on $e\setminus\{v\}$, we have $\psi(v)\notin L(v)$.

Given a partial coloring $Q=(U,L,\phi)$, a pair $(e,\psi)\in E\times \F$ is \emphd{active} in $Q$ if
\begin{itemize}
    \item $|e\cap U|\geq2$;
    \item $\psi(u)=\phi(u)$ for every $u\in e\setminus U$;
    \item $\psi(u)\in L(u)$ for every $u\in e\cap U$.
\end{itemize}

For $v\in U$, $c\in L(v)$ and $1\leq\ell\leq r$, set
\begin{align*}
    d_{\ell,Q}(v,c)
    &=\bigl|\{(e,\psi)\text{ active} \;:\;  v\in e\cap U,\; |e\cap U|=\ell+1,\; \psi(v)=c\}\bigr|,\\
    D_{\ell,Q}&=\max_{v\in U,\; c\in L(v)}d_{\ell,Q}(v,c).
\end{align*}

A coloring of $U$ from the lists $L$ that avoids all active forbidden colorings on their residual supports extends a compatible partial coloring to an $(L_0,\F)$-coloring. Indeed, edges with no uncolored vertices are already satisfied, singleton obstructions are excluded by compatibility, and any violation on a larger residual support would come from an active pair.

We always retain the original edge and forbidden coloring in an active pair. Deleting colors and restricting to uncolored vertices preserve the matching condition for each original edge. More explicitly, two distinct forbidden colorings on an edge disagree on every remaining coordinate; if any coordinate has already been colored, at most one forbidden coloring on that edge can agree with the partial coloring. 

For $e\in E$ and $w\in e$, let
\[
 S_e(w)=\{w\}\cup\bigcup_{\substack{f\in E\setminus\{e\}\\w\in f}}(f\setminus\{w\}).
\]
Linearity gives $S_e(w)\cap e=\{w\}$. For distinct $w,w'\in e$, the sets $S_e(w)$ and $S_e(w')$ are disjoint: an intersection through the same external edge contradicts linearity, and an intersection through two distinct external edges gives a Berge triangle with $e$.

For $v\in e$, put
\[
 R_v(e)=\bigcup_{w\in e\setminus\{v\}}S_e(w).
\]
In particular, $e\setminus\{v\}\subseteq R_v(e)$ and $v\notin R_v(e)$.
For distinct edges $e,e'$ containing $v$, the sets $R_v(e)$ and $R_v(e')$ are disjoint. To see this, let $w\in e\setminus\{v\}$ and $w'\in e'\setminus\{v\}$. An intersection involving one of these two vertices gives a Berge triangle (or contradicts linearity). Otherwise, the intersection lies on external edges $f\ni w$ and $f'\ni w'$. Neither external edge can be $e$ or $e'$. If $f=f'$, then $e,f,e'$ give a Berge triangle. If $f\neq f'$, then $e,f,f',e'$ give a Berge $4$-cycle with distinct vertices $v,w,z,w'$, where $z$ is the intersection vertex. These separation properties remain valid after restricting to uncolored vertices and active pairs.

We also use three standard probabilistic tools.

\begin{theo}[Talagrand's inequality~{\cite[Section 3.1]{Molly2014telegrand}}]
\label{thm:talagrand}
Let $X$ be a non-negative integer-valued random variable, not identically zero, determined by independent trials $T_1,\ldots,T_n$.
Assume that, for some $\mu,\rho>0$, the following two conditions hold.
\begin{enumerate}[label=\textup{(T\arabic*)}]
    \item Changing the outcome of any one trial $T_i$ can change $X$ by at most $\mu$.
    \item For every integer $s\geq 1$, if $X\geq s$, then there is a set of at most $\rho s$ trials certifying that $X\geq s$.
\end{enumerate}
Then, for every $t\geq0$,
\[
\Pr\left[
    |X-\E X|
    \geq t+20\mu\sqrt{\rho\E X}+64\mu^2\rho
\right]
\leq
4\exp\left(
    -\frac{t^2}{8\mu^2\rho(\E X+t)}
\right).
\]
\end{theo}
We use the following standard consequence of the multiplicative Chernoff bound \cite[Appendix~A, Theorem~A.1.12]{Alon2016Prob}.
\begin{lemma}[Chernoff's inequality]
\label{lem:chernoff}
Let $X$ be a sum of finitely many independent Bernoulli variables.
For every $M\geq\E X$ with $M > 0$ and every $0<u\leq1$,
\[
 \P(X>(1+u)M)\leq\exp(-u^2M/3).
\]
\end{lemma}

\begin{lemma}[Lov\'asz Local Lemma~{\cite[Corollary 5.1.2]{Alon2016Prob}}]
\label{lem:lll}
Let $\mathcal A$ be a finite family of events, each having probability at most $p$.
Suppose that every $A\in\mathcal A$ is jointly independent of all but at most $d$ other events in $\mathcal A$.
If $ep(d+1)\leq 1$, then
\[
\P\left(\bigcap_{A\in\mathcal A} A^c\right)>0.
\]
\end{lemma}

Finally, we introduce some basic properties about polynomials we would use.

For a polynomial $P(z)=\sum_{m=1}^r b_mz^m$, we denote by $P^{(\ell)}$ its $\ell$th derivative. For $1\leq\ell\leq r$, set
\[
 P_\ell(z)=\frac{P^{(\ell)}(z)}{\ell!}
         =\sum_{m=\ell}^r\binom m\ell b_mz^{m-\ell}.
\]
Thus $P_\ell(z)$ is the coefficient of $h^\ell$ in $P(z+h)$ and
\begin{align}
 P(z+h)-P(z)&=\sum_{\ell=1}^r P_\ell(z)h^\ell,\label{eq:taylor}\\
 P_\ell(z+h)&=\sum_{m=\ell}^r\binom m\ell P_m(z)h^{m-\ell}.\label{eq:translation}
\end{align}

\section{One-step coloring}\label{sec:onestep}

\begin{lemma}\label{lem:nibble}
For every integer $r\geq1$, there exist $K=K(r)>0$ and $q_0=q_0(r)>0$ such that the following holds for all integers $q\geq q_0$.
Set
\[
    a=(16r2^{r-1})^{-1},\qquad
    h=(\log q)^{-2},\qquad \zeta=h^2.
\]
For any uncrowded hypergraph $G=(V,E)$ of rank at most $r+1$ and a $q$-fold cover $(L_0,\F)$ of $G$, suppose $Q_i=(U_i,L_i,\phi_i)$ is a compatible partial coloring satisfying 
\[
|L_i(v)|=s_i,\; v\in U_i,\qquad (\log q)^{12}\leq s_i\leq q,\qquad D_{\ell,Q_i}\leq A_{\ell,i}\defeq (x_i s_i)^\ell P_\ell(t_i) ,\;1\leq\ell\leq r.
\]
for some
\[
0\leq t_i\leq2,\qquad a/\log q<x_i\leq1
\]
and a polynomial $P(z)$ such that
\[
    P(z)=\sum_{\ell=1}^r b_\ell z^\ell,\qquad b_\ell\geq1,\qquad P(1)\leq \log q,\qquad 
\]
Define
\[
 p_i=\frac{h}{x_i s_i},\qquad
 \gamma_i=\prod_{\ell=1}^r(1-p_i^\ell)^{A_{\ell,i}}.
\]
Then there exists a compatible partial coloring
$Q_{i+1}=(U_{i+1},L_{i+1},\phi_{i+1})$ such that 
\[
U_{i+1}\subseteq U_i,\qquad L_{i+1}(v)\subseteq L_i(v),\;v\in U_{i+1},\qquad \varphi_{i+1}(v)=\varphi_i(v),\;v\in V\setminus U_i.
\]
and
\begin{align}\label{eq:Li+1}
    |L_{i+1}(v)|=s_{i+1},\; v\in U_{i+1},\qquad s_{i+1}
    :=\lfloor(1-\zeta)\gamma_i s_i\rfloor.
\end{align}
Moreover, for
\[
t_{i+1}=t_i+h,\qquad x_{i+1}=x_i-h+Kh/\log q
\]
we have
\begin{align}
     D_{\ell,Q_{i+1}}&\leq(x_{i+1}s_{i+1})^\ell P_\ell(t_{i+1}),\; 1\leq\ell\leq r,\label{eq:Dl}\\
     (1-2&\zeta)\gamma_i s_i\leq s_{i+1}\leq\gamma_i s_i,\label{eq:si+1gammaisi}\\
     P(t_i+h)-P(t_i)&\leq-\log\gamma_i\leq(1+p_i)\left((P(t_i+h)-P(t_i)\right).\label{eq:loggammai}
\end{align}
\end{lemma}

Now we begin to prove \Cref{lem:nibble}.
Set
\[
 \theta_i:=p_i s_i=\frac{h}{x_i},\qquad
 \tau_i:=P(t_i+h)-P(t_i)=\sum_{\ell=1}^r A_{\ell,i}p_i^\ell,
 \qquad \sigma_i:=1-\theta_i+4\theta_i\tau_i+4p_i.
\]
The equality defining $\tau_i$ follows from~\eqref{eq:taylor}.
The random procedure is given in Algorithm~\ref{alg:nibble}.

\begin{algorithm}[htbp]
\caption{One-step nibble: a randomized update of the partial coloring}\label{alg:nibble}
\small
\noindent\textbf{Input:} An uncrowded hypergraph $G$ with cover $(L_0,\F)$; a compatible partial coloring $Q_i=(U_i,L_i,\varphi_i)$ with $|L_i(u)|=s_i$ for $u\in U_i$; and the parameters $p_i$, $\theta_i$, $\gamma_i$, $\zeta$ defined above. Put $s_{i+1}=\lfloor(1-\zeta)\gamma_i s_i\rfloor$.
\par\smallskip
\noindent\textbf{Output:} Either \textsc{Failure} (a failed random trial, not a claim of uncolorability), or a compatible partial coloring $Q_{i+1}=(U_{i+1},L_{i+1},\varphi_{i+1})$ extending $Q_i$, with $|L_{i+1}(u)|=s_{i+1}$ for every $u\in U_{i+1}$.
\begin{enumerate}[label=(\arabic*),leftmargin=2em,itemsep=2pt]
\item \textbf{Tentative choices.} Independently for every $u\in U_i$, sample $Y_u\in L_i(u)\cup\{\bot\}$ with
\[
 \P(Y_u=\bot)=1-\theta_i,\qquad
 \P(Y_u=c)=p_i\quad(c\in L_i(u)).
\]
The value $\bot$ means that $u$ makes no tentative color choice.
\item \textbf{Safety tests.} For every $u\in U_i$ and $c\in L_i(u)$, let $\Safe(u,c)$ be the condition that there is \emph{no} active pair $(e,\psi)$ in $Q_i$ with $u\in e\cap U_i$, $\psi(u)=c$, and
\[
 Y_w=\psi(w)\qquad\text{for every }w\in (e\cap U_i)\setminus\{u\}.
\]
All safety tests are evaluated from the tentative choices $Y$, before any colors are accepted.
\item \textbf{Acceptance.} Define the set of \emphd{accepted vertices}
\[
 \mathsf{Acc}_i=\{u\in U_i:Y_u\ne\bot\text{ and }\Safe(u,Y_u)\}.
\]
Permanently color each $u\in\mathsf{Acc}_i$ with $Y_u$. Every vertex in $U_i\setminus\mathsf{Acc}_i$ is \emphd{unaccepted} and remains uncolored. Set
\[
 U_{i+1}=U_i\setminus\mathsf{Acc}_i,\qquad
 \varphi_{i+1}=\varphi_i\cup\{(u,Y_u):u\in\mathsf{Acc}_i\}.
\]
\item \textbf{Equalizing coins.} For every $u\in U_i$ and $c\in L_i(u)$, compute from the input state
\[
 \pi_i(u,c)=\prod_{\ell=1}^r(1-p_i^\ell)^{d_{\ell,Q_i}(u,c)},
\]
then independently sample $\xi_{u,c}\in\{0,1\}$, also independently of $Y$, with $\P(\xi_{u,c}=1)=\gamma_i/\pi_i(u,c)$. Define the temporary lists
\[
 \widehat L_i(u)=\{c\in L_i(u):\Safe(u,c)\text{ and }\xi_{u,c}=1\}\quad(u\in U_i).
\]
The coins do not affect acceptance; only temporary lists at unaccepted vertices enter the output.
\item \textbf{Failure test and return.} If $|\widehat L_i(u)|<s_{i+1}$ for some \emph{unaccepted} vertex $u\in U_{i+1}$, output \textsc{Failure} and stop. Otherwise choose any $L_{i+1}(u)\subseteq\widehat L_i(u)$ of size $s_{i+1}$ for each $u\in U_{i+1}$, and return $Q_{i+1}=(U_{i+1},L_{i+1},\varphi_{i+1})$.
\end{enumerate}
\end{algorithm}

It is easy to see that any outcome $Q_{i+1}$ is a compatible partial coloring and
\[
U_{i+1}\subseteq U_i,\qquad L_{i+1}(v)\subseteq L_i(v),\;v\in U_{i+1},\qquad \varphi_{i+1}(v)=\varphi_i(v),\;v\in V\setminus U_i.
\]
Now we prove (\ref{eq:Li+1}) to (\ref{eq:loggammai}).

Note that for any vertex $u\in U_i$ and $c\in L_i(u)$, since $G$ is uncrowded, no two edges incident to $u$ have another common vertices, and by independence we have
\begin{align}\label{eq:safe}
    \P(\Safe(u,c))=\prod_{\ell=1}^r(1-p_i^\ell)^{d_{\ell,Q_i}(u,c)}\geq\gamma_i.
\end{align}
Thus the algorithm makes sense, and
\begin{align}\label{eq:EL_i}
    \P(c\in\widehat L_i(u))=\gamma_i,\; \forall\; c\in L_i(u),\quad \text{ and }\quad \E[|\widehat L_i(u)|]=\gamma_is_i.
\end{align}
This flip coin process is designed to make the probability of survival of any color equals, which simplifies computations later.

Before we move on, we make some estimations on the parameters for large enough $q_0$:
\begin{align}
    \theta_i=\dfrac{h}{x_i}&\leq\frac1{a\log q},\qquad p_i\leq\frac1{a(\log q)^{13}},\\
     \tau_i&\leq hP'(t_i+h)\leq hP'(3)
 \leq\frac{r3^{r-1}}{\log q}.\label{eq:smalltau}
\end{align}
For $0\leq u\leq1/2$, we have $u\leq-\log(1-u)\leq u+u^2$. Applying this with $u=p_i^\ell$ proves~\eqref{eq:loggammai}.

 In particular, $\gamma_i=1-O_r\left((\log q)^{-1}\right)$.
Also
\[
 \sigma_i=1-\theta_i\left(1-4\tau_i-4/s_i\right),
\]
so, for large $q_0$,
\begin{equation}\label{eq:largesigma}
 1-\theta_i\leq\sigma_i\leq1,\qquad \gamma_i\geq\frac12,\qquad \sigma_i\gamma_i\geq\frac12.
\end{equation}
Finally,
\[
 s_{i+1}-(1-2\zeta)\gamma_i s_i\geq\zeta\gamma_i s_i-1 \geq\tfrac12(\log q)^8-1\geq0,
\]
which proves~\eqref{eq:si+1gammaisi}.

We next bound the probability that a fixed active pair produces a constraint of a prescribed residual size after one nibble round.
Fix an active pair $(e,\psi)$ with $|e\cap U_i|=m+1$, a vertex $v\in e\cap U_i$ and a subset $S\subseteq (e\cap U_i)\setminus \set{v}$ of size $\ell$, where $1\leq \ell\leq m$. Set $T=(e\cap U_i)\setminus (\set{v}\cup S)$. Then we have the following lemma.

\begin{lemma}\label{lem:residualedgeprob}
    For every value of $Y_v$ of positive probability,
    \[
    \P\left(
    \begin{gathered}
    Y_u=\psi(u),\;\forall \;u\in T,\\
    w\in U_{i+1},\; \psi(w)\in\widehat L_i(w),\;\forall \;w\in S
    \end{gathered}
    \,\middle|\,Y_v
    \right)
    \leq p_i^{m-\ell}(\sigma_i\gamma_i)^\ell.
    \]
\end{lemma}

\begin{proof}
    Let $M_0=\{Y_u=\psi(u)\text{ for every }u\in T\}$. Then
    \[
    \P(M_0 \mid Y_v)=p_i^{m-\ell}.
    \]
    For the rest of the proof, also condition on $M_0$; all probabilities below use this conditioning unless stated otherwise.
   
  Say that an original edge $f$ \emphd{blocks} $(w,c)$ if an active pair $(f,\psi')$ satisfies $\psi'(w)=c$ and $Y_u=\psi'(u)$ for every $u\in (f \cap U_i)\setminus\{w\}$.
For $w\in S$, let $A_{w,1}$ be the event that $\xi_{w,\psi(w)}=1$ and no edge other than $e$ blocks $(w,\psi(w))$. Let $A_{w,2}$ be the event that either $Y_w=\bot$ or some edge other than $e$ blocks $(w,Y_w)$, and put $A_w=A_{w,1}\cap A_{w,2}$.

    The trials at vertices of $S_e(w)\setminus\{w\}$ are unaffected by our conditioning. Moreover, $A_{w,1}$ is independent of $Y_w$. Removing the factor associated with $e$ from~\eqref{eq:safe} gives
\[
 \P(A_{w,1})=\frac{\gamma_i}{1-p_i^m}.
\]
For each fixed value $Y_w\neq\bot$, a union bound over external active pairs bounds the probability of an external obstruction by $\tau_i$. Hence
\begin{align*}
 \P(A_w)
 &\leq(1-\theta_i)\frac{\gamma_i}{1-p_i^m}+\theta_i\tau_i\\
 &\leq\gamma_i(1-\theta_i+2p_i+2\theta_i\tau_i)
 \leq\gamma_i(\sigma_i-2p_i).
\end{align*}
Here we used $p_i^m\leq p_i\leq1/2$ and $\gamma_i\geq1/2$.
The event $A_w$ is determined by $Y_w$, the trials at the external vertices in $S_e(w)$, and the coin $\xi_{w,\psi(w)}$; external blocking of $Y_w$ may involve any available color. By the branch separation in Section~\ref{sec:pre}, these sets of trials are disjoint for different $w\in S$. Thus the events $A_w$ are independent under the stated conditioning.

Let $M$ be the event that $e$ blocks $(w,Y_w)$ for at least one $w\in S$.
If $Y_v=\bot$, then $M$ is impossible. Otherwise, $M$ forces every residual vertex of $e$, including the blocked vertex, to choose the value prescribed by a forbidden coloring $\psi'$ on $e$ with $\psi'(v)=Y_v$. The matching condition gives at most one such $\psi'$. It need not equal $\psi$, but it prescribes the tentative values at all $\ell$ vertices of $S$. Therefore
\[
 \P(M\mid M_0,Y_v)\leq p_i^\ell.
\]
If every $w\in S$ is unaccepted and keeps $\psi(w)$, then either every $A_w$ holds or $M$ holds. No independence between $M$ and the $A_w$ is needed. A union bound gives the desired probability at most
\[
 p_i^{m-\ell}
 \left[\bigl(\gamma_i(\sigma_i-2p_i)\bigr)^\ell+p_i^\ell\right]
 \leq p_i^{m-\ell}(\sigma_i\gamma_i)^\ell.
\]
For the last inequality, note that $\sigma_i-2p_i\geq0$ for large $q_0$, that $\alpha^\ell+\beta^\ell\leq(\alpha+\beta)^\ell$ for $\alpha,\beta\geq0$, and that
$\gamma_i(\sigma_i-2p_i)+p_i\leq\gamma_i\sigma_i$.
\end{proof}

Fix $v\in U_i$, $c\in L_i(v)$, and $1\leq\ell\leq r$. For each active pair $(e,\psi)$ with $\psi(v)=c$, let $W_{e,\ell}(v,c)$ to be the Bernoulli variable indicating the event that exactly $\ell$ vertices of $(e\cap U_i)\setminus \{v\}$ are not colored and $\psi(u)$ are kept in $\widehat L_i(u)$ for these vertices $u$, while all other vertices in $(e\cap U_i)\setminus \{v\}$ are colored identically to $\psi$.
The matching condition makes $\psi$ unique for a given $e,v,c$. Set
\[
 X_\ell(v,c)=\sum_{(e,\psi)\text{ active}:\,\psi(v)=c}W_{e,\ell}(v,c).
\]
This variable is defined without requiring $v$ to remain uncolored or to retain $c$. For a nonfailing outcome, if $v\in U_{i+1}$ and $c\in L_{i+1}(v)$, then
\[
 d_{\ell,Q_{i+1}}(v,c)\leq X_\ell(v,c).
\]

Since each edge $e$ containing $v$ corresponds to at most one $\psi\in \F$ such that $\psi(v)=c$, by \Cref{lem:residualedgeprob} and (\ref{eq:translation}), for every value of $Y_v$ with positive probability,
\begin{align}
 \E[X_\ell(v,c)\mid Y_v]
 &\leq(\sigma_i\gamma_i)^\ell
 \sum_{m=\ell}^r\binom m\ell p_i^{m-\ell}d_{m,Q_i}(v,c)\notag\\
 &\leq(\sigma_i\gamma_i x_i s_i)^\ell
 \sum_{m=\ell}^r\binom m\ell P_m(t_i)h^{m-\ell}\notag\\
 &=(\sigma_i\gamma_i x_i s_i)^\ell P_\ell(t_i+h)
 =:\mu_\ell.
 \label{eq:mu}
\end{align}

We next establish concentration bounds for the temporary list sizes $|\widehat L_i(v)|$ and the variables $X_\ell(v,c)$.

\begin{claim}\label{clm:concentration}
For every $v\in U_i$, $c\in L_i(v)$, and $1\leq\ell\leq r$, with $\mu_\ell$ as in~\eqref{eq:mu}, 
    \begin{align*}
    \P\bigl(|\widehat L_i(v)|<(1-\zeta)\gamma_i s_i\bigr)&<4\exp\left(-\frac{(\log q)^4}{1024r}\right),\\
    \P\bigl(X_\ell(v,c)>(1+\zeta)^\ell\mu_\ell\bigr)
 &\leq\exp(-a(\log q)^3/6).
    \end{align*}
\end{claim}

\begin{proof}
For $v\in U_i$, let $Z_v=s_i-|\widehat L_i(v)|$. Then
$\E Z_v=(1-\gamma_i)s_i$ by~\eqref{eq:EL_i}.
Changing $Y_v$ has no effect on $Z_v$. For $u\neq v$, changing $Y_u$ changes $Z_v$ by at most two: by linearity there is at most one original edge containing $u$ and $v$, and the old and new values of $Y_u$ belong to at most two forbidden colorings on that edge. Changing a coin changes $Z_v$ by at most one.
Each deleted color can be certified by one failed coin, or by the at most $r$ tentative values on an active pair that makes the color unsafe. Thus $Z_v\geq z$ has a certificate of at most $rz$ trials.

Apply Theorem~\ref{thm:talagrand} with $\mu=2$ and $\rho=r$ (the case $Z_v\equiv0$ is immediate). Since $\E Z_v\leq s_i$, $\gamma_i\geq1/2$, and $\zeta\sqrt{s_i}\geq(\log q)^2$, for large $q_0$ we have
\[
 40\sqrt{r\E Z_v}+256r\leq\frac{\zeta\gamma_i s_i}{2}.
\]
Using $t=\zeta\gamma_i s_i/2$ in that theorem, we obtain
\begin{equation}\label{eq:listtail}
 \P\bigl(|\widehat L_i(v)|<(1-\zeta)\gamma_i s_i\bigr)
 \leq4\exp\left(-\frac{\zeta^2s_i}{1024r}\right)
 \leq4\exp\left(-\frac{(\log q)^4}{1024r}\right).
\end{equation}

Conditional on $Y_v$, the summands of $X_\ell(v,c)$ are independent. Indeed, the summand associated with $e$ is determined by $Y_v$, the tentative values at uncolored vertices in $R_v(e)$, and coins at vertices of $e \cap U_i\setminus\{v\}$. The sets $R_v(e)$ are disjoint for different original edges through $v$, by Section~\ref{sec:pre}. We condition only on $Y_v$, not on whether $v$ is accepted or keeps $c$.
Since $(1+\zeta)^\ell\geq1+\zeta$, Lemma~\ref{lem:chernoff} and~\eqref{eq:mu} yield
\[
 \P\bigl(X_\ell(v,c)>(1+\zeta)^\ell\mu_\ell\mid Y_v\bigr)
 \leq\exp(-\zeta^2\mu_\ell/3).
\]
Every coefficient of $P$ is at least one, so $P_\ell(t_i+h)\geq1$. By~\eqref{eq:largesigma},
\[
 \mu_\ell\geq(x_i s_i/2)^\ell
 \geq\left(\frac{a(\log q)^{11}}2\right)^\ell
 \geq\frac{a(\log q)^{11}}2
\]
for sufficiently large $q_0$. Averaging over $Y_v$ gives
\begin{equation}\label{eq:degreetail}
 \P\bigl(X_\ell(v,c)>(1+\zeta)^\ell\mu_\ell\bigr)
 \leq\exp(-a(\log q)^3/6).
\end{equation}

\end{proof}

We use \Cref{lem:lll} to simultaneously control $X_l(v,c)$.

Let $\Gamma=(U_i,E_i)$ be the graph on $U_i$ in which two vertices are adjacent when they lie in a common active pair.
Since $G$ has rank at most $r+1$, the maximum degree of $\Gamma$
\[
 \Delta(\Gamma)\leq \max_{v\in U_i}\sum_{c\in L_i(v)}\sum_{\ell=1}^r \ell d_{\ell,Q_i}(v,c) \leq r s_i\sum_{\ell=1}^r A_{\ell,i}.
\]
Since $1\leq x_i s_i\leq q$,
\begin{align*}
    \sum_{\ell=1}^r A_{\ell,i}\leq (x_i s_i)^r\sum_{\ell=1}^r P_\ell(t_i)=(x_i s_i)^r\bigl(P(t_i+1)-P(t_i)\bigr)\leq 3^r q^r\log q.
\end{align*}
It follows that
\[
\Delta(\Gamma)\leq r3^r q^{r+1}\log q.
\]

For the bad event $|\widehat L_i(v)|<(1-\zeta)\gamma_i s_i$ and $X_{\ell}(v,c)>(1+\zeta)^\ell\mu_\ell$, they involves only values of $Y_u$, $\zeta_{u,c}$ for $u$ at distance at most two to $v$ in $\Gamma$.
Thus each bad event is jointly independent to all bad events except for bad events concerning vertices at distance at most four to some vertex in $\Gamma$.
Further, there exists at most $(1+rs_i)$ bad events for each vertex $v$, and thus each bad event is jointly independent to all bad events except for
\[
D=(1+rs_i)(1+\Delta(\Gamma)+\Delta(\Gamma)^2+\Delta(\Gamma)^3+\Delta(\Gamma)^4)=\exp(O_r(\log q))
\]
bad events.
By \eqref{eq:listtail} and \eqref{eq:degreetail},
each bad event has probability at most $\rho=\exp(-c_r(\log q)^3)$.
For large enough $q_0$, we have $e\rho(D+1)< 1$, and \Cref{lem:lll} gives an outcome of Algorithm~\ref{alg:nibble} such that all bad events do not happen.

For this outcome $Q_{i+1}$, (\ref{eq:Li+1}) follows directly.
For (\ref{eq:Dl}), it remains to prove
\[
 (1+\zeta)\sigma_i\gamma_i x_i s_i\leq x_{i+1}s_{i+1}.
\]

Note that for large enough $q_0$, $(1+\zeta)/(1-2\zeta)\leq1+6\zeta$.
Using $\sigma_i\leq 1$ and $s_{i+1}\geq (1-2\zeta)\gamma_i s_i$,
we have
\[
\frac{(1+\zeta)\sigma_i\gamma_i x_i s_i}{s_{i+1}}\leq \dfrac{1+\zeta}{1-2\zeta}\sigma_ix_i\leq (1+6\zeta)\sigma_ix_i\leq \sigma_ix_i+6\zeta x_i.
\]
By the definition of $\sigma_i$,
\[
\sigma_ix_i=x_i-h+4h\tau_i+4h/s_i.
\]
For large enough $K$,
\[
4h\tau_i\leq Kh/(3\log q),\quad 4h/s_i\leq Kh/(3\log q),\quad 6\zeta x_i \leq Kh/(3\log q).
\]
Thus
\[
\frac{(1+\zeta)\sigma_i\gamma_i x_i s_i}{s_{i+1}}\leq x_{i+1}.
\]
And we complete the proof.

\section{Proof of the main result}\label{sec:main}
After the nibble, we use a counting argument of the type introduced by Rosenfeld~\cite{Rosenfeld2020color}, which was applied by Martinsson~\cite{Martinsson2021Rosen} to the Johansson--Molloy theorem.

\begin{lemma}\label{lem:Rosenfeld}
Let $r,s\geq1$ be integers, let $(L_0,\F)$ be a cover of a hypergraph $G=(V,E)$ of rank at most $r+1$, and let $Q=(U,L,\varphi)$ be a compatible partial coloring with $|L(v)|=s$ for every $v\in U$. If
\[
 \sum_{\ell=1}^r2^\ell\frac{D_{\ell,Q}}{s^\ell}\leq\frac12,
\]
then there are at least $(s/2)^{|U|}$ $(L_0,\F)$-colorings of $G$ extending $\varphi$ and using the lists $L$ on $U$.
\end{lemma}

\begin{proof}
For $W\subseteq U$, let $n(W)$ count the colorings of $W$ from $L$ that avoid every active forbidden coloring whose residual support is contained in $W$. Call these colorings \emphd{feasible}, and set $n(\varnothing)=1$. We prove by induction on $|W|$ that
\[
 n(W)\geq\frac{s}{2}n(W\setminus\{v\})\qquad(v\in W).
\]
The case $|W|=1$ is immediate, since every active pair has at least two uncolored vertices.

There are $s n(W\setminus\{v\})$ extensions obtained by assigning a color from $L(v)$ to $v$. An infeasible extension must realize an active pair $(e,\psi)$ with $v\in R:=e\cap U\subseteq W$.
For a fixed such pair, the colors on $R$ are prescribed by $\psi$, so restriction to $W\setminus R$ is an injection from these bad extensions into the feasible colorings of $W\setminus R$. Their number is therefore at most $n(W\setminus R)$. If $|R|=\ell+1$, applying the induction hypothesis successively within $W\setminus\{v\}$ gives
\[
 n(W\setminus R)\leq\left(\frac{s}{2}\right)^{-\ell}n(W\setminus\{v\}).
\]
There are at most $sD_{\ell,Q}$ active pairs of residual size $\ell+1$ containing $v$. Hence
\begin{align*}
 n(W) &\geq s\left(1-\sum_{\ell=1}^r D_{\ell,Q}\left(\frac{s}{2}\right)^{-\ell}\right) n(W\setminus\{v\})\\
 &\geq\frac{s}{2}n(W\setminus\{v\}).
\end{align*}
Iterating yields $n(U)\geq(s/2)^{|U|}$. By compatibility, every coloring counted by $n(U)$ extends $\varphi$ to an $(L_0,\F)$-coloring of $G$.
\end{proof}

Now we are ready to prove \Cref{thm:main}.

\begin{proof}[Proof of Theorem~\ref{thm:main}]
Fix a $q$-fold cover $(L_0,\F)$ of $G=(V,E)$ satisfying \eqref{eq:colordegree}. Let $K$ be given by \Cref{lem:nibble}, and set
\[
 a=(16r2^{r-1})^{-1},\qquad h=(\log q)^{-2},\qquad \zeta=h^2,\qquad P(z)=\sum_{\ell=1}^r\left(\frac{D_\ell}{q^\ell}+1\right)z^\ell.
\]
We increase $q_0=q_0(r)$ as needed throughout the proof. The hypothesis gives
\[
 P(1)\leq\log q-20\log\log q+r\leq\log q.
\]
Define, for every integer $i\geq0$,
\[
 t_i=ih,\qquad x_i=1-ih+\frac{iKh}{\log q},
\]
and let $N$ be the least integer such that $x_N\leq a/\log q$. For sufficiently large $q$, we have $\log q>2K$ and $h<a/\log q$. Since
\[
 x_i-x_{i+1}=h\left(1-\frac{K}{\log q}\right)\in(0,h),
\]
minimality of $N$ gives
\[
 0<\frac{a}{\log q}-h<x_N\leq\frac{a}{\log q},\qquad Nh<\frac{1}{1-K/\log q} \leq1+\frac{2K}{\log q}<2.
\]
In particular, $N<2(\log q)^2$, $0\leq t_i<2$ for $i\leq N$, and $a/\log q<x_i\leq1$ for $i<N$.

The coefficients of $P$ are nonnegative, so
\[
 P'(z)\leq r2^{r-1}\log q \qquad(0\leq z\leq2).
\]
Consequently, for $k\leq N$,
\[
 P(t_k)\leq P(1)+2Kr2^{r-1}\leq\log q-20\log\log q+C_r, \qquad C_r:=r+2Kr2^{r-1}.
\]
Also $P(t_k)\leq P(2)\leq2^r\log q$.

Start with $Q_0=(V,L_0,\varnothing)$ and $s_0=q$. The initial degree bounds hold because
\[
 (x_0s_0)^\ell P_\ell(t_0)=q^\ell+D_\ell\geq D_{\ell,Q_0}.
\]
Also $s_0=q\geq(\log q)^{12}$ for sufficiently large $q$. We construct $Q_1,\ldots,Q_N$ inductively with lists of common size $s_i$, maintaining
\[
 D_{\ell,Q_i}\leq(x_is_i)^\ell P_\ell(t_i),\qquad (\log q)^{12}\leq s_i\leq q;
\]
if all vertices are colored at any stage, we are done.

Suppose $Q_0,\ldots,Q_{k-1}$ have been constructed, where $1\leq k\leq N$. The preceding bounds allow us to apply \Cref{lem:nibble} at $Q_{k-1}$, obtaining $Q_k$ with the required degree bounds and $s_k\leq q$.
It remains to check the lower bound on $s_k$.
For every $j<k$, the induction hypothesis gives
\[
 p_j=\frac{h}{x_js_j}\leq\frac{1}{a(\log q)^{13}}.
\]
Using \eqref{eq:si+1gammaisi}, \eqref{eq:loggammai}, $P(0)=0$, and $\log(1-2\zeta)\geq-4\zeta$, we obtain
\begin{align*}
 \log s_k &\geq\log q+k\log(1-2\zeta)-\sum_{j=0}^{k-1}(1+p_j)\bigl(P(t_{j+1})-P(t_j)\bigr)\\
 &\geq\log q-4k\zeta-\left(1+\frac{1}{a(\log q)^{13}}\right)P(t_k)\\
 &\geq20\log\log q-C_r-\frac{8}{(\log q)^2}-\frac{2^r}{a(\log q)^{12}}\\
 &\geq12\log\log q
\end{align*}
for sufficiently large $q$.
Thus $s_k\geq(\log q)^{12}$, and the induction closes, including $k=N$.

Finally,
\[
 t_N+2x_N\leq1+\frac{2K+2a}{\log q}\leq2.
\]
By the degree bounds and \eqref{eq:taylor},
\begin{align*}
 \sum_{\ell=1}^r 2^\ell\frac{D_{\ell,Q_N}}{s_N^\ell}
 &\leq\sum_{\ell=1}^r(2x_N)^\ell P_\ell(t_N)\\
 &=P(t_N+2x_N)-P(t_N)\\
 &\leq2x_Nr2^{r-1}\log q\\
 &\leq2ar2^{r-1}=\frac18<\frac12.
\end{align*}
Applying \Cref{lem:Rosenfeld} completes the coloring.
\end{proof}

\begin{proof}[Proof of Corollary~\ref{cor:main}]
For sufficiently large $\Delta$, put
\[
 q=\left\lceil
 \left(1+22\frac{\log\log\Delta}{\log\Delta}\right)
 \left(\frac{r\Delta}{\log(\Delta+1)}\right)^{1/r}
 \right\rceil.
\]
Then $q\geq q_0(r)$. Writing $L=\log\Delta$ and $M=\log L$, we have
\[
 \log q=\frac Lr-\frac Mr+\frac{\log r}{r}+o_r(1),\qquad
 \log\log q=M-\log r+o_r(1),\qquad
 \frac{\Delta}{q^r}=\frac Lr-22M+o_r(1).
\]
The ceiling affects these expansions by $o_r(1)$. Hence
\[
 \log q-20\log\log q-\frac{\Delta}{q^r}
 =\left(2-\frac1r\right)\log\log\Delta+O_r(1)>0.
\]
For every $q$-fold cover of $H$, the matching condition gives $D_r\leq\Delta$, and $D_\ell=0$ for $\ell<r$. Theorem~\ref{thm:main} therefore colors every such cover, proving the corollary.
\end{proof}

\begin{proof}[Proof of Corollary~\ref{cor:color-degree}]
Restrict the lists to size
$q=\lceil(1+\varepsilon)(rd/\log d)^{1/r}\rceil$ and keep only forbidden colorings using available colors. Then $D_r\leq d$ and $D_\ell=0$ for $\ell<r$. As $d\to\infty$,
\[
 \frac d{q^r}\leq\frac{\log d}{r(1+\varepsilon)^r},\qquad
 \log q=\frac1r\log d-\frac1r\log\log d+O_{r,\varepsilon}(1).
\]
The difference $\log q-d/q^r$ is a positive constant multiple of $\log d$, up to an $O_{r,\varepsilon}(\log\log d)$ term. It therefore exceeds $20\log\log q$ for sufficiently large $d$, and Theorem~\ref{thm:main} applies.
\end{proof}

\section{Concluding remarks}\label{sec:conclude}
The exclusion of Berge 4-cycles is used in the conditional independence of the contributions from distinct original edges through a fixed vertex. Without this exclusion, the determining branches may intersect, and the degree concentration in Claim~\ref{clm:concentration} no longer follows from independence. Extending the criterion to linear hypergraphs without Berge triangles, but with Berge 4-cycles, would require additional control of these dependencies. We conjecture that the exclusion of Berge 4-cycles is unnecessary for the conclusion of \Cref{thm:main}.

\begin{conj}\label{conj:triangle-free}
For every fixed integer $r\geq 2$, there exists a constant
$C_r>0$ such that the following holds for all sufficiently
large integers $q$.

Let $G$ be a linear hypergraph of rank at most $r+1$
containing no Berge triangles, and let $(L_0,\mathcal F)$
be a $q$-fold DP-cover of $G$. For $1\leq\ell\leq r$,
let $D_\ell$ denote the maximum color-degree associated
with edges of size $\ell+1$.

If
\[
\sum_{\ell=1}^r \frac{D_\ell}{q^\ell}
\leq \log q-C_r\log\log q,
\]
then $(L_0,\mathcal F)$ admits a proper DP-coloring.
\end{conj}

In particular, this would imply that every linear
Berge-triangle-free $(r+1)$-uniform hypergraph $G$
of maximum degree at most $\Delta$ satisfies
\[
\chi_{\mathrm{DP}}(G)
\leq (1+o_r(1))
\left(\frac{r\Delta}{\log\Delta}\right)^{1/r}
\qquad\text{as }\Delta\to\infty.
\]

Palette sparsification of uniform hypergraphs is also investigated recently in \cite{czumaj2026streaminghypergraphcoloringpalette,casselgren2026singleconflictcoloringpalette,DMV26v3}. We think it may not be difficult to prove the following conjecture and we leave it for the readers. 

\begin{conj}[Uniform DP palette sparsification]\label{conj:palette}
Fix $r\geq1$ and $\varepsilon>0$. There is $C=C(r,\varepsilon)>0$ such that for every $\beta\in(0,1)$, all sufficiently large $\Delta$ in terms of $r,\varepsilon,\beta$, every $n$-vertex uncrowded $(r+1)$-uniform hypergraph of maximum degree at most $\Delta$ with any $Q$-fold correspondence cover has the following property. If $n\geq2$, $m\leq Q$, and
\begin{equation}\label{eq:palette-parameters}
 Q\geq(1+\varepsilon)\left(\frac{r\Delta}{\beta\log\Delta}\right)^{1/r},
 \qquad m\geq\Delta^{\beta/r}+C(\log n)^{1/(r+1)},
\end{equation}
then independently sampled uniform $m$-element sublists admit a coloring with probability at least $1-1/n$.
\end{conj}

\section*{AI declaration}

The authors used ChatGPT 5.6 Sol and 6 Astra for language polishing, organization and presentation of the manuscript and informal discussion of minor auxiliary details in the computation of probability concentrations. All other mathematical arguments are due to the authors. The authors take full responsibility for the contents.

\printbibliography

@article {Molly2014telegrand,
    AUTHOR = {Molloy, Michael and Reed, Bruce},
     TITLE = {Colouring graphs when the number of colours is almost the
              maximum degree},
   JOURNAL = {J. Combin. Theory Ser. B},
  FJOURNAL = {Journal of Combinatorial Theory. Series B},
    VOLUME = {109},
      YEAR = {2014},
     PAGES = {134--195},
      %ISSN = {0095-8956,1096-0902},
   MRCLASS = {05C15},
  MRNUMBER = {3269906},
MRREVIEWER = {Daqing\ Yang},
       %DOI = {10.1016/j.jctb.2014.06.004},
       %URL = {https://doi.org/10.1016/j.jctb.2014.06.004},
}

@book {Alon2016Prob,
    AUTHOR = {Alon, Noga and Spencer, Joel H.},
     TITLE = {The Probabilistic Method},
    SERIES = {Wiley Series in Discrete Mathematics and Optimization},
   EDITION = {Fourth Edition},
 PUBLISHER = {John Wiley \& Sons, Inc., Hoboken, NJ},
      YEAR = {2016},
     %PAGES = {xiv+375},
      %ISBN = {978-1-119-06195-3},
   MRCLASS = {60-02 (05C80 05D40 60C05 60F10 60G42)},
  MRNUMBER = {3524748},
}

@article {Rosenfeld2020color,
    AUTHOR = {Rosenfeld, Matthieu},
     TITLE = {Another approach to non-repetitive colorings of graphs of
              bounded degree},
   JOURNAL = {Electron. J. Combin.},
  FJOURNAL = {Electronic Journal of Combinatorics},
    VOLUME = {27},
      YEAR = {2020},
    NUMBER = {3},
     PAGES = {Paper No. 3.43, 16},
      %ISSN = {1077-8926},
   MRCLASS = {05C15},
  MRNUMBER = {4245156},
MRREVIEWER = {Daniel\ Kr\'al},
       %DOI = {10.37236/9667},
       %URL = {https://doi.org/10.37236/9667},
}

@misc{Martinsson2021Rosen,
      title={A simplified proof of the Johansson--Molloy Theorem using the Rosenfeld counting method}, 
      author={Anders Martinsson},
      year={2021},
      eprint={2111.06214},
      archivePrefix={arXiv},
      primaryClass={math.CO},
      %url={https://arxiv.org/abs/2111.06214}, 
}

@techreport{Johansson,
  author      = {Johansson, Anders},
  title       = {Asymptotic choice number for triangle free graphs},
  journal={unpublished manuscript},
  number      = {91--95},
  year        = {1996}
}

@article {Molloy19list,
    AUTHOR = {Molloy, Michael},
     TITLE = {The list chromatic number of graphs with small clique number},
   JOURNAL = {J. Combin. Theory Ser. B},
  FJOURNAL = {Journal of Combinatorial Theory. Series B},
    VOLUME = {134},
      YEAR = {2019},
     PAGES = {264--284},
    %  ISSN = {0095-8956,1096-0902},
   MRCLASS = {05C15 (05C69)},
  MRNUMBER = {3906639},
MRREVIEWER = {Hsin-Hao\ Lai},
      % DOI = {10.1016/j.jctb.2018.06.007},
      % URL = {https://doi.org/10.1016/j.jctb.2018.06.007},
}

@article {Postle2018corres,
    AUTHOR = {Dvo\v r\'ak, Zden\v ek and Postle, Luke},
     TITLE = {Correspondence coloring and its application to list-coloring
              planar graphs without cycles of lengths 4 to 8},
   JOURNAL = {J. Combin. Theory Ser. B},
  FJOURNAL = {Journal of Combinatorial Theory. Series B},
    VOLUME = {129},
      YEAR = {2018},
     PAGES = {38--54},
      %ISSN = {0095-8956,1096-0902},
   MRCLASS = {05C10 (05C15)},
  MRNUMBER = {3758240},
MRREVIEWER = {Deming\ Li},
       %DOI = {10.1016/j.jctb.2017.09.001},
       %URL = {https://doi.org/10.1016/j.jctb.2017.09.001},
}

@article {Anton2019DP,
    AUTHOR = {Bernshteyn, Anton and Kostochka, Alexandr},
     TITLE = {D{P}-colorings of hypergraphs},
   JOURNAL = {European J. Combin.},
  FJOURNAL = {European Journal of Combinatorics},
    VOLUME = {78},
      YEAR = {2019},
     PAGES = {134--146},
      %ISSN = {0195-6698,1095-9971},
   MRCLASS = {05C15 (05C65)},
  MRNUMBER = {3916024},
       %DOI = {10.1016/j.ejc.2019.01.011},
       %URL = {https://doi.org/10.1016/j.ejc.2019.01.011},
}

@article {Ajtai1982uncrowded,
    AUTHOR = {Mikl{\'o}s Ajtai and J{\'a}nos Koml{\'o}s and J{\'a}nos Pintz and Joel Spencer and Endre Szemer{\'e}di},
     TITLE = {Extremal uncrowded hypergraphs},
   JOURNAL = {J. Combin. Theory Ser. A},
  FJOURNAL = {Journal of Combinatorial Theory. Series A},
    VOLUME = {32},
      YEAR = {1982},
    NUMBER = {3},
     PAGES = {321--335},
     % ISSN = {0097-3165,1096-0899},
   MRCLASS = {05C65},
  MRNUMBER = {657047},
MRREVIEWER = {F.\ Sterboul},
      % DOI = {10.1016/0097-3165(82)90049-8},
       %URL = {https://doi.org/10.1016/0097-3165(82)90049-8},
}

@article {Anton2019Molloy,
    AUTHOR = {Bernshteyn, Anton},
     TITLE = {The {J}ohansson--{M}olloy theorem for {DP}-coloring},
   JOURNAL = {Random Structures Algorithms},
  FJOURNAL = {Random Structures \& Algorithms},
    VOLUME = {54},
      YEAR = {2019},
    NUMBER = {4},
     PAGES = {653--664},
      %ISSN = {1042-9832,1098-2418},
   MRCLASS = {05C15 (05D40)},
  MRNUMBER = {3957361},
MRREVIEWER = {Niranjan\ Balachandran},
       %DOI = {10.1002/rsa.20811},
       %URL = {https://doi.org/10.1002/rsa.20811},
}

@article {Ajtai1980ramsey,
    AUTHOR = {Ajtai, Mikl\'os and Koml\'os, J\'anos and Szemer\'edi, Endre},
     TITLE = {A note on {R}amsey numbers},
   JOURNAL = {J. Combin. Theory Ser. A},
  FJOURNAL = {Journal of Combinatorial Theory. Series A},
    VOLUME = {29},
      YEAR = {1980},
    NUMBER = {3},
     PAGES = {354--360},
      %ISSN = {0097-3165,1096-0899},
   MRCLASS = {05C55 (05C35)},
  MRNUMBER = {600598},
MRREVIEWER = {J.\ E.\ Graver},
       %DOI = {10.1016/0097-3165(80)90030-8},
       %URL = {https://doi.org/10.1016/0097-3165(80)90030-8},
}

@article {Mubayi13color,
    AUTHOR = {Frieze, Alan and Mubayi, Dhruv},
     TITLE = {Coloring simple hypergraphs},
   JOURNAL = {J. Combin. Theory Ser. B},
  FJOURNAL = {Journal of Combinatorial Theory. Series B},
    VOLUME = {103},
      YEAR = {2013},
    NUMBER = {6},
     PAGES = {767--794},
     % ISSN = {0095-8956,1096-0902},
   MRCLASS = {05C15 (05C65 05D40)},
  MRNUMBER = {3127593},
MRREVIEWER = {Dmitry\ A.\ Shabanov},
      % DOI = {10.1016/j.jctb.2013.09.003},
      % URL = {https://doi.org/10.1016/j.jctb.2013.09.003},
}

@misc{Postle25chrom,
      title={The chromatic number of triangle-free hypergraphs}, 
      author={Lina Li and Luke Postle},
      year={2025},
      eprint={2202.02839},
      archivePrefix={arXiv},
      primaryClass={math.CO},
     % url={https://arxiv.org/abs/2202.02839}, 
}

@article {Vu2002localsparse,
    AUTHOR = {Vu, Van H.},
     TITLE = {A general upper bound on the list chromatic number of locally
              sparse graphs},
   JOURNAL = {Combin. Probab. Comput.},
  FJOURNAL = {Combinatorics, Probability and Computing},
    VOLUME = {11},
      YEAR = {2002},
    NUMBER = {1},
     PAGES = {103--111},
      %ISSN = {0963-5483,1469-2163},
   MRCLASS = {05C15 (05C35)},
  MRNUMBER = {1888186},
MRREVIEWER = {Andr\'as\ Gy\'arf\'as},
      % DOI = {10.1017/S0963548301004898},
       %URL = {https://doi.org/10.1017/S0963548301004898},
}

@incollection {Kang2023survey,
    AUTHOR = {Kang, Dong Yeap and Kelly, Tom and K\"uhn, Daniela and Methuku, Abhishek and Osthus, Deryk},
     TITLE = {Graph and hypergraph colouring via nibble methods: a survey},
 BOOKTITLE = {European {C}ongress of {M}athematics},
     PAGES = {771--823},
 PUBLISHER = {EMS Press, Berlin},
      YEAR = {2023},
      %ISBN = {978-3-98547-051-8; 978-3-98547-551-3},
   MRCLASS = {05C15 (05D40)},
  MRNUMBER = {4615769},
       %DOI = {10.4171/8ECM/11},
       %URL = {https://doi.org/10.4171/8ECM/11},
}

@misc{dhawan2026ind,
      title={The independence number of uncrowded hypergraphs: bounds matching the shattering threshold}, 
      author={Abhishek Dhawan and Abhishek Methuku and Minh-Quan Vo},
      year={2026},
      eprint={2606.18048v1},
      archivePrefix={arXiv},
      primaryClass={math.CO},
      %url={https://arxiv.org/abs/2606.18048}, 
      sorttitle = {01},
}

@misc{dhawan2026list,
      title={List coloring uncrowded hypergraphs at the shattering threshold}, 
      author={Abhishek Dhawan and Abhishek Methuku and Minh-Quan Vo},
      year={2026},
      eprint={2606.18048v2},
      archivePrefix={arXiv},
      primaryClass={math.CO},
      %url={https://arxiv.org/abs/2606.18048}, 
      sorttitle = {02},
       note = {Version 2, submitted on 7 September 2026}
}

@incollection {Ili21,
    AUTHOR = {Iliopoulos, F.},
     TITLE = {Improved bounds for coloring locally sparse hypergraphs},
 BOOKTITLE = {Approximation, randomization, and combinatorial optimization.
              {A}lgorithms and techniques},
    SERIES = {LIPIcs. Leibniz Int. Proc. Inform.},
    VOLUME = {207},
     PAGES = {Art. No. 39, 16},
 PUBLISHER = {Schloss Dagstuhl. Leibniz-Zent. Inform., Wadern},
      YEAR = {2021},
   MRCLASS = {68R10 (68Q87)},
  MRNUMBER = {4366594},
}

@misc{yu2026list,
      title={The list coloring number of uncrowded hypergraphs}, 
      author={Jing Yu and Junchi Zhang},
      year={2026},
      eprint={2607.05256},
      archivePrefix={arXiv},
      primaryClass={math.CO},
      note = {Version 1, submitted on 6 July 2026}
      %url={https://arxiv.org/abs/2607.05256}, 
}

@misc{czumaj2026streaminghypergraphcoloringpalette,
      title={Streaming Hypergraph Coloring via Palette Sparsification}, 
      author={Artur Czumaj and Pan Peng and Ruizhe Shi and Christian Sohler},
      year={2026},
      eprint={2608.19623},
      archivePrefix={arXiv},
      primaryClass={math.CO},
      %url={https://arxiv.org/abs/2608.19623}, 
}

@misc{casselgren2026singleconflictcoloringpalette,
      title={Single conflict coloring and palette sparsification of uniform hypergraphs}, 
      author={Carl Johan Casselgren and Kalle Eriksson},
      year={2026},
      eprint={2601.06990},
      archivePrefix={arXiv},
      primaryClass={math.CO},
      %url={https://arxiv.org/abs/2601.06990}, 
}

@misc{DMV26v3,
  title         = {List coloring uncrowded hypergraphs at the shattering threshold},
  author        = {Abhishek Dhawan and Abhishek Methuku and Minh-Quan Vo},
  year          = {2026},
  eprint        = {2606.18048v3},
  archivePrefix = {arXiv},
  primaryClass  = {math.CO},
  note          = {Version 3, submitted on 24 October 2026}
}
\end{document}